\documentclass[12pt]{article}
\usepackage[utf8]{inputenc}
\usepackage{amsmath,amssymb,amsthm}
\usepackage[inline]{enumitem}
\usepackage[nocompress]{cite}
\usepackage{color}

\title{Lipschitz bijections between boolean functions}
\author{Tom Johnston\thanks{Mathematical Institute, University of Oxford, Oxford, OX2 6GG, UK \texttt{\{thomas.johnston,scott\}@maths.ox.ac.uk}}\hphantom{ } and Alex Scott\footnotemark[1] \thanks{Supported by a Leverhulme Trust Research Fellowship.}}
\date{}

\newtheorem{theorem}{Theorem}

\newtheorem{lemma}[theorem]{Lemma}

\newtheorem{problem}{Problem}
\newtheorem{claim}{Claim}

\newtheorem{question}{Question}

\newtheorem{corollary}[theorem]{Corollary}

\DeclareMathOperator{\xor}{XOR}
\DeclareMathOperator{\maj}{Majority}
\DeclareMathOperator{\dic}{Dictator}
\DeclareMathOperator{\avgStretch}{avgStretch}
\DeclareMathOperator{\dist}{dist}
\DeclareMathOperator{\bin}{Bin}
\DeclareMathOperator{\dlip}{diss}
\DeclareMathOperator{\poly}{poly}

\newcommand{\lipsize}[1]{\|#1 \|_{{\rm Lip}}}

\newcommand{\floor}[1]{\left\lfloor #1 \right\rfloor}
\newcommand{\ceil}[1]{\left\lceil #1 \right\rceil}
\newcommand{\expec}[1]{\mathbb{E} \left[ #1 \right]}

\newcommand{\vara}{a}
\newcommand{\varb}{b}
\newcommand{\varc}{c}
\newcommand{\vard}{d}

\newcommand{\eps}{\epsilon}

\begin{document}
		\maketitle
		\begin{abstract}
			We answer four questions from a recent paper of Rao and Shinkar \cite{RAO2018} on Lipschitz bijections between functions from $\{0,1\}^n$ to $\{0,1\}$. 
			\begin{enumerate*}[label={(\arabic*)}]
				\item We show that there is no $O(1)$-bi-Lipschitz bijection from $\dic$ to $\xor$ such that each output bit depends on $O(1)$ input bits.
				\item We give a construction for a mapping from $\xor$ to $\maj$ which has  average stretch $O(\sqrt{n})$, matching a previously known lower bound.
				\item We give a $3$-Lipschitz embedding $\phi : \{0,1\}^n \to \{0,1\}^{2n+1}$ such that $\xor(x) = \maj(\phi(x))$ for all $x \in \{0,1\}^n$.
				\item We show that with high probability there is a $O(1)$-bi-Lipschitz mapping from $\dic$ to a uniformly random balanced function.
			\end{enumerate*}
		\end{abstract}
		\section{Introduction}	
		Given two boolean functions $f,g : \{0,1\}^n \to \{0,1\}$ we say that a bijection $\phi: \{0,1\}^n \to \{0,1\}^n$ is a \emph{mapping from $f$ to $g$} if, for every $x \in \{0,1\}^n$, we have $f(x) = g(\phi(x))$. The analysis of boolean functions, in particular their Fourier coefficients and noise stability, is widely-studied and has applications in many areas of mathematics including the theory of social choice, complexity theory and in the hardness of approximations (see for instance \cite{bourgain2002distribution,kindler2018gaussian, lovett2011bounded, Mossel2010, boppana1997average, bernasconi1998mathematical, dinur2005hardness, ben1989collective,kahn1988influence, benjamini1999noise}).
		A frequent theme in the literature is the analysis of the similarities and differences between boolean functions with different geometric or structural properties; for example, between functions such as $\dic$ that are
		determined by a small number of coordinates, and functions such as
		$\maj$ or $\xor$ that are not. One measure of similarity between functions is the existence of a Lipschitz mapping (with small constant) between them. In this paper we continue the study of Rao and Shinkar \cite{RAO2018} on Lipschitz mappings between the boolean functions $\dic, \xor$, $\maj$ and a uniformly random balanced function and answer several of the questions they pose.
				
		Write a point $x  \in \{0,1\}^n$ as $x = (x_1, \dots, x_n)$ and, for $\phi : \{0,1\}^n \to \{0,1\}^n$, write $\phi = \left( \phi_1, \dots, \phi_n \right)$. For $x, y \in \{0,1\}^n$, let $|x| = \sum_{i=1}^n x_i$ denote the Hamming weight of $x$ and $\dist(x,y) = \sum_{i=1}^n |x_i - y_i|$ the Hamming distance between $x$ and $y$.
		
		A bijection $\phi: \{0,1\}^n \to \{0,1\}^n$ is said to be \emph{$C$-Lipschitz} if, for all $x,y \in \{0,1\}^n$, $\dist(\phi(x), \phi(y)) \leq C \dist(x,y)$, and $\phi$ is said to be \emph{$C$-bi-Lipschitz} if both $\phi$ and $\phi^{-1}$ are $C$-Lipschitz. As a relaxation from being Lipschitz we define the \emph{average stretch} of a mapping $\phi: \{0,1\}^n \to \{0,1\}^n$ by \[ \avgStretch(\phi) = \mathbb{E}_{x,i}\left[ \dist(\phi(x), \phi(x+e_i))\right]\] where $x \in \{0,1\}^n$ and $i \in [n]$ are both chosen independently and uniformly at random.
		
		Given a bijection $\phi : \{0,1\}^n \to \{0,1\}^n$ we say that the $j$th output bit $\phi_j$ \emph{depends on} the $i$th input bit $x_i$ if there exists $x \in \{0,1\}^n$ such that $\phi_j(x) \neq \phi_j(x + e_i)$. If every output bit depends on at most $k$ input bits, we say the map $\phi$ is \emph{$k$-local}.
		
		In Sections \ref{sec:dictoxor} and \ref{sec:xortomaj} we study mappings between three boolean functions $\dic$, $\xor$ and $\maj$ which we define by
		\begin{itemize}
			\item $\dic(x) = x_1$,
			\item $\xor(x) = \sum_{i=1}^n x_i \mod 2$,
			\item $\maj(x) = 1$ if $|x| > n/2$ and $\maj(x) = 0$ otherwise.
		\end{itemize}
		In Section $\ref{sec:random}$ we consider a uniformly random balanced function $f : \{0,1\}^n \to \{0,1\}$, where we say a boolean function $g: \{0,1\}^n \to \{0,1\}$ is \emph{balanced} if $g^{-1}(1)$ and $g^{-1}(0)$ are of the same size. Clearly, both $\xor$ and $\dic$ are always balanced while $\maj$ is only balanced when $n$ is odd (and so a bijection from $\dic$ or $\xor$ to $\maj$ can only exist for odd $n$).

		\subsection{Bijections between $\dic$ and $\maj$}

		The functions $\dic$ and $\maj$ are in many ways opposites of each other. For example, the first coordinate clearly has a large influence on the value of $\dic$, while for $\maj$ every coordinate has the same small amount of influence. There are many results which show that functions which differ from $\maj$ in some way must have influential coordinates and are therefore similar to $\dic$ functions. For example, the ``Majority Is Stablest'' theorem of Mossel, O'Donnell and Oleszkiewicz in \cite{Mossel2010} shows that if a balanced boolean function is essentially more noise-stable than $\maj$, then it must have an influential coordinate.
		
		It is straightforward to see that no bijection $\phi$ from $\dic$ to $\maj$ can be $C$-Lipschitz for any $C < n/2$. Indeed, suppose $\phi$ is such a bijection and let $x \in \{0,1\}^n$ be such that $\phi(x) = \sum_{i=1}^n e_i$. Clearly $y := x - e_1$ differs from $x$ in the first coordinate so $\dic(y) \neq \dic(x)$. By the definition of $\phi$, we must have $\maj(\phi(y))  = 0$ and $|\phi(y)|$ is at most $n/2$. Hence, $\dist(\phi(x), \phi(y)) \geq n/2$ which gives a contradiction.
		
		For maps from $\maj$ to $\dic$, the situation is better. Rao and Shinkar \cite{RAO2018} showed that, for all odd $n$, there is a mapping $\psi$ from $\maj$ to $\dic$ which is 11-Lipschitz.
		As noted above, the function $\psi^{-1}$ cannot be $O(1)$-Lipschitz. However, it has the weaker property that $\psi^{-1}$  has $O(1)$-average stretch. Indeed, Rao and Shinkar's construction gives a Lipschitz function that is in fact an $O(1)$-bi-Lipschitz bijection from the upper half of $\{0,1\}^n$ to the half-cube with first coordinate $1$; similarly $\psi$ induces a $O(1)$-bi-Lipschitz bijection between the lower half of $\{0,1\}^n$ and the half-cube with first coordinate $0$. As there are only $2^{n-1}$ edges between the two half-cubes, the average stretch of $\psi^{-1}$ is $O(1)$.

		\subsection{Bijections between $\dic$ and $\xor$}
		Rao and Shinkar note that the map $\phi: \{0,1\}^n \to \{0,1\}^n$ given by $\phi(x) = (\xor(x), x_2, x_3, \dots, x_n)$ is a mapping from $\dic$ to $\xor$, and it is easy to check that $\phi$ is 2-bi-Lipschitz. This leads them to consider maps with stronger properties. In the above map, the first output bit depends on all $n$ input bits and the map is not $k$-local for any $k<n$. However, one can easily find a map which is $2$-local: Rao and Shinkar give the example $\phi(x) = (x_1 + x_2, x_2 + x_3, \dots, x_{n-1} + x_n, x_n)$ which is $2$-Lipschitz and $2$-local, although the inverse of this map is not even $(n-1)$-Lipschitz. In \cite{RAO2018} Rao and Shinkar construct a map $\phi$ which is $2$-Lipschitz and $3$-local and where the inverse is $O(\log n)$-Lipschitz. This leads them to ask the following question.
		
		\begin{question}[Question 6.1 in \cite{RAO2018}]
			Is there a mapping from Dictator to XOR that is $O(1)$-local and $O(1)$-bi-Lipschitz?
		\end{question}
	
		We answer this in the negative with the following theorem.
	
		\begin{theorem}\label{thm:dictoxor}
		Let $\phi: \{0,1\}^n \to \{0,1\}^n$ be a mapping from $\dic$ to $\xor$ which is $C$-Lipschitz and where each output bit depends on at most $k$ input bits. Then there is a constant $\delta = \Omega(1/(k + \log(C)))$ such that the inverse map $\phi^{-1}$ is not $ \delta \log(n)$-Lipschitz.
		
		Furthermore, if $\phi$ is a linear mapping, then we may take $\delta = \Omega(1/\log(C+k))$.
		\end{theorem}
			
		It follows that the map constructed by Rao and Shinkar in \cite{RAO2018} is essentially best possible: if $\phi$ is a mapping from $\dic$ to $\xor$ which is $O(1)$-Lipschitz and $O(1)$-local, then $\phi^{-1}$ cannot be $o(\log n)$-Lipschitz. We prove Theorem \ref{thm:dictoxor} in Section \ref{sec:dictoxor}.
		 
		\subsection{Bijections between $\xor$ and $\maj$}
		Since we have a 2-bi-Lipschitz map from $\dic$ to $\xor$, we might expect the maps between $\xor$ and $\maj$ to behave similarly to those between $\dic$ and $\maj$. Indeed, composing this 2-bi-Lipschitz map from $\dic$ to $\xor$ with the $O(1)$-Lipschitz map from $\maj$ to $\dic$ gives a $O(1)$-Lipschitz map from $\maj$ to $\xor$, and the argument used to show that there is no $O(1)$-Lipschitz mapping from $\dic$ to $\maj$ can also be applied to mappings from $\xor$ to $\maj$ to show that they cannot be $C$-Lipschitz for $C < n/2$. However, while there is a map from $\xor$ to $\dic$ with constant average stretch, Rao and Shinkar \cite{RAO2018} show that for any mapping $\phi$ from $\xor$ to $\maj$ the average stretch is $\Omega(\sqrt{n})$. They then ask if this lower bound is tight:
		\begin{question}[Question 6.3 in \cite{RAO2018}]
			Is there a mapping $\phi$ from $\xor$ to $\maj$ such that $\avgStretch(\phi) = O(\sqrt{n})$?
		\end{question}

		In Section \ref{sec:xortomaj} we show that the bound is indeed tight by giving a map which has average stretch $O(\sqrt{n})$. This result also follows from a recent result of Boczkowski and Shinkar \cite{boczkowski2019mappings} where they prove, for any two sets $A,B \subseteq \{0,1\}^n$ with $|A| = |B| = 2^{n-1}$, there is a mapping $\phi: A \to B$ with $\expec{\dist(x, \phi(x))} \leq \sqrt{2n}$.
		
		\begin{theorem}\label{thm:xortomaj}
			For odd $n$ there is a mapping $\phi$ from $\xor$ to $\maj$ such that \[\avgStretch(\phi) = O\left( \sqrt{n} \right).\]
		\end{theorem}
		
		Given the need for such a mapping to have large average stretch, Rao and Shinkar also ask what happens if we relax the problem from finding a bijection to finding an embedding in a larger space.
		\begin{question}[Question 6.4 in \cite{RAO2018}]
			Is there a Lipschitz embedding $\phi: \{0,1\}^n \to \{0,1\}^{\poly(n)}$ such that $\xor(x) = \maj(\phi(x))$ for all $x \in \{0,1\}^n$?
		\end{question}
		
		We give a simple construction that gives a positive answer:
		
		\begin{theorem}\label{thm:embed}
			For every $n$ there exists a 3-Lipschitz embedding $\phi: \{0,1\}^n \to \{0,1\}^{2n+1}$ such that $\xor(x) = \maj(\phi(x))$ for all $x \in \{0,1\}^n$.
		\end{theorem}
	
		\subsection{Bijections from $\dic$ to a random $f$}
	
		Rao and Shinkar also consider random functions. We say a function $f:\{0,1\}^n \to \{0,1\}$ is \emph{balanced} if $f^{-1}(1)$ and $f^{-1}(0)$ are of the same size.  Building on a construction of H{\aa}stad, Leighton and Newman \cite{Hastad1987}, Rao and Shinkar show that with high probability there is a bijection from $\dic$ to a uniformly random balanced function $f$ which has average stretch bounded by an absolute constant \cite{RAO2018}. They ask whether we can in fact ask for more:
		
		\begin{question}[Question 6.2 in \cite{RAO2018}]
			Is it true that with high probability there is a $O(1)$-bi-Lipschitz mapping from $\dic$ to a uniformly random balanced function $f$?
		\end{question}
		
		In Section \ref{sec:random} we give a positive answer to this question. In fact, we prove a stronger statement about the maximum distance between $x$ and $\phi(x)$.
	
		\begin{theorem}\label{thm:random}
			Let $f: \{0,1\}^n \to \{0,1\}$ be a balanced boolean function chosen uniformly at random. Then, with high probability, there exists a mapping $\phi: \{0,1\}^n \to \{0,1\}^n$ from $\dic$ to $f$ such that for all $x \in \{0,1\}^n$, \[ \dist(x, \phi(x)) \leq C \]
			where $C$ is an absolute constant. In particular, $\phi$ is $(2C + 1)$-bi-Lipschitz.
		\end{theorem}
		
		It is straightforward to extend this to bijections between two independent uniformly random balanced boolean functions $f$ and $g$.
	
		\begin{corollary}\label{cor}
			Let $f, g : \{0,1\}^n \to \{0,1\}$ be two balanced boolean functions chosen independently and uniformly at random. Then, with high probability, there exists a mapping $\phi: \{0,1\}^n \to \{0,1\}^n$ from $f$ to $g$ such that for all $x \in \{0,1\}^n$,\[ \dist(x, \phi(x)) \leq C\] where $C$ is an absolute constant. In particular, $\phi$ is $(2C + 1)$-bi-Lipschitz.
		\end{corollary}
		\begin{proof}
			By the above theorem there exists a mapping $\phi_f$ from $\dic$ to $f$ such that $\dist(x, \phi_f(x)) \leq c$ and a mapping $\phi_g$ from $\dic$ to $g$ such that $\dist(x, \phi_g(x)) \leq c$ for some absolute constant $c$, both with high probability. Using the triangle inequality, it is easy to see that $\phi = \phi_g \circ \phi_f^{-1}$ is a mapping from $f$ to $g$ such that $\dist(x, \phi(x)) \leq 2c$. This proves the first part of the corollary with $C = 2c$.
			
			The second part also follows easily. Fix $x \in \{0,1\}^n$ and let $y = x + e_i$ for some $i$. Then 
			\[\dist(\phi(x), \phi(y)) \leq \dist(x, \phi(x)) + \dist(x,y) + \dist(y, \phi(y)) \leq 2C + 1 \]
			and the triangle inequality shows $\phi$ is $(2C+1)$-Lipschitz. A similar argument shows that $\phi^{-1}$ is also $(2C+1)$-Lipschitz.
		\end{proof}
	
	The rest of this paper is organised as follows. In Section \ref{sec:dictoxor} we prove Theorem \ref{thm:dictoxor}, in Section \ref{sec:xortomaj} we prove Theorem \ref{thm:xortomaj} and Theorem \ref{thm:embed}. We then prove Theorem \ref{thm:random} in Section \ref{sec:random}.
	
		\section{Bijections from $\dic$ to $\xor$}\label{sec:dictoxor}
				
		Given a mapping $\phi: \{0,1\}^n \to \{0,1\}^n$ from $\dic$ to $\xor$ we define the \emph{dependency graph} $D_\phi$ as follows. Let $D_\phi$ be the bipartite graph with vertex sets $A = \{a_1, \dots, a_n\}$ and $B = \{b_1, \dots, b_n\}$ where there is an edge $a_ib_j$ if and only if the $j$th output bit of $\phi$ depends on the $i$th input bit. The following lemma shows that if there is an output bit $\phi_j$ which is at a large distance from $a_1$ in the dependency graph, then changing $\phi_j$ must cause many input bits to change. 
		\begin{lemma}\label{lem:distance}
			Suppose $\phi$ is a mapping from $\dic$ to $\xor$ such that the distance between $a_1$ and $b_j$ in $D_\phi$ is at least $d$. Then, for any $y \in \{0,1\}^n$,
			\[ \dist\left(\phi^{-1}(y), \phi^{-1}(y + e_j) \right) \geq \frac{d+1}{2}. \]
		\end{lemma}
		\begin{proof}
			Suppose that $\phi^{-1}(y) = x$ and $\phi^{-1}(y + e_j) = x'$ and let the set of coordinates in which they differ be $X$. Now let $D_\phi'$ be the subgraph of $D_\phi$ induced by $A_X := \{ a_i : i \in X\}$ and its neighbours. Since $\phi(x)$ and $\phi(x')$ differ in the $j$th bit, one of the neighbours of $b_j$ must be in $A_X$ and $b_j \in V(D_\phi')$. As $\phi$ is a mapping from $\dic$ to $\xor$ and $\xor(y) \neq \xor(y + e_j)$, we must have $\dic(x) \neq \dic(x')$ and so $a_1 \in A_X$. If $a_1$ and $b_j$ are in the same component of $D_\phi'$, then $D_{\phi}'$ must contain a path from $a_1$ to $b_j$, which must have length at least $d$. Since $D'_\phi$ is bipartite, $|X| \geq (d+1)/2$ and we are done.
			
			Otherwise, let the component of $D'_\phi$ containing $a_1$ have vertices $A_Y \cup B_{Y'}$, where $Y, Y' \subset [n]$. Consider the input $x + e_Y$ where $e_Y = \sum_{i \in Y} e_i$: we will show that $\phi(x + e_Y) = \phi(x)$. If $y' \in Y'$, then the output bit $\phi_{y'}$ depends only on the input bits in $Y$ and input bits that are not in $X$ so $\phi_{y'}(x + e_Y)  = \phi_{y'}(x') = \phi_{y'}(x)$ (since $\phi(x)$ and $\phi(x')$ differ only in the $j$th bit and $j \not \in Y'$). Changing the value of $x$ on $Y$ can only change $\phi(x)$ in the output bits in $Y'$ and so $\phi_i(x) = \phi_i(x + e_Y)$ for every $i \not \in Y'$. This means that $\phi(x) = \phi(x+e_Y)$, which gives a contradiction as $\phi$ is a bijection.	
		\end{proof}
	
		Now that we have related distance in the dependency graph $D_\phi$ to the Hamming distance of the inverse map, we prove Theorem \ref{thm:dictoxor} by showing that the conditions on $\phi$ imply that there is a vertex at least logarithmically far from $a_1$ in $D_\phi$.
		\begin{proof}[Proof of Theorem \ref{thm:dictoxor}]
			As each output bit of $\phi(x)$ depends on at most $k$ input bits of $x$, the degree of a vertex $b_j$ in $D_\phi$ is at most $k$.
			
			Let us now show that the degrees of the $a_i$ are also bounded. Fix a vertex $a_i$ in $D_\phi$ and let it have neighbours $b_{j_1}, \dots, b_{j_m}$. The value of the output bit $j_1$ depends on $i$ and at most $k-1$ other bits which we denote $X$. By definition, there is some assignment $x$ of the $k-1$ bits $X$ such that $\phi(x)$ and $\phi(x + e_i)$ differ in the bit $j_1$. Let $U \in \{0,1\}^n$ be a bit string chosen uniformly at random: then the probability that this is equal to $x$ on $X$ is $2^{1-k}$ and, hence, the probability that $\phi(U)$ and $\phi(U+e_i)$ differ in bit $j_1$ is at least $2^{1-k}$.
			The same holds for bits $j_2, \dots, j_m$ and so, by the linearity of expectation, the expected number of bits in which $\phi(U)$ and $\phi(U+e_i)$ differ is at least $m 2^{1-k}$. There must be some value for $U$ for which the number of bits that differ is at least $m2^{1-k}$ and, as $\phi$ is $C$-Lipschitz, we must have $m2^{1-k} \leq C$. In particular, the degree of a vertex $a_i$ is bounded by $C 2^{k-1}$.
			
			Define $\Delta := \max \{k, C2^{k-1} \}$ so that every vertex has degree at most $\Delta$. The number of vertices at distance at most $l$ from $a_1$ is at most $(\Delta + 1)^l$ and hence, if the distance from $a_1$ to any vertex is at most $d$, $(\Delta + 1)^d \geq 2n$ and $d \geq (\log 2n) / \log(\Delta + 1)$. As the graph is bipartite, there must be a vertex $b_j$ at distance at least $d-1$ from $a_1$ and hence, by Lemma \ref{lem:distance}, $\phi$ is not $\delta \log(n)$-Lipschitz for $\delta = 1/(2\log(\Delta + 1))$.
			
			If $\phi$ is a linear map, then $\phi(x)$ and $\phi(x + e_i)$ differ in the bits $j_1, \dots, j_m$ for every $x$. Hence, the degree of a vertex $a_i$ is bounded by $C$ (compared to $C2^{k-1}$ in the general case), and the bound on $\delta$  follows from bounding the maximum degree by $C + k$.
			\end{proof}
		
		The construction presented by Rao and Shinkar in \cite{RAO2018} gives a linear map which shows $\delta(2,k) = O(1/\log k)$.
		Theorem \ref{thm:dictoxor} shows that this is tight for linear maps, but only gives the bound $\delta(2,k) = \Omega(1/k)$ for general maps. This raises the following question: Does there exist a map $\phi$ from $\dic$ to $\xor$ which is $C$-Lipschitz, where each output bit depends on at most $k$ input bits and such that $\phi^{-1}$ is $o(\log n/ \log k)$-Lipschitz as $n, k \to \infty$?

		\section{Bijections from $\xor$ to $\maj$}\label{sec:xortomaj}
		It was shown by Rao and Shinkar in \cite{RAO2018} that any mapping $\phi: \{0,1\}^n \to \{0,1\}^n$ from $\xor$ to $\maj$ must have average stretch $\Omega\left( \sqrt{n} \right)$ in any direction. In this section we prove Theorem \ref{thm:xortomaj} by constructing a map with average stretch $O(\sqrt{n})$. Our strategy is to map each $x \in \{0,1\}^n$ for which $\xor(x) = \maj(x)$ to itself, and otherwise to swap $x$ with a $y$ chosen such that $|x| + |y| = n$. The problem is then to find a matching, so that when the matched elements are switched, they are not too far from their neighbours on average. We do this by matching according to a symmetric chain decomposition so that the Hamming distance between switched elements is minimised.  This is similar to the proof of Theorem 1.2 in \cite{RAO2018} which defines a mapping from $\maj$ to $\dic$ by permuting elements along the chains in a particular symmetric chain decomposition. Recall that a \emph{symmetric chain} is a path $C = \{c_k ,\dots, c_{n-k}\}$ in the hypercube such that $|c_i| = i$. A \emph{symmetric chain decomposition} is a partition of the hypercube into symmetric chains. It is well known that these decompositions exist for all $n$ (see e.g. \cite{bollobas1986combinatorics}).

		\begin{proof}[Proof of Theorem \ref{thm:xortomaj}]
			Suppose we have a symmetric chain decomposition of $\{0,1\}^n$ and, for a point $x \in \{0,1\}^n$, let $y_x$ be the unique point in the chain containing $x$ such that $|x| + |y_x| = n$. We note that $\xor(x) \neq \xor(y_x)$ and $\maj(x) \neq \maj(y_x)$, so the bijection $\phi : \{0,1\}^n \to \{0,1\}^n$ defined by
				\[ \phi(x) = \begin{cases}
			x & \text{ if } \xor(x) = \maj(x)\\
			y_x & \text{ if } \xor(x) \neq \maj(x)\\
			\end{cases}\]
			is a mapping from $\xor$ to $\maj$.
			
			\begin{claim}\label{claim}
				Suppose $x \in \{0,1\}^n$ has Hamming weight $m$. Then, for any $i$, 
				\[ \dist\left( \phi(x), \phi(x+ e_i) \right) \leq |n - 2m| + 3 \]
			\end{claim}
			\begin{proof}[Proof of Claim \ref{claim}]
				The proof of this claim is a case analysis over $\xor(x)$, $\maj(x)$, whether $e_i \leq x$ and, where necessary, $\maj(x+e_i)$. Let us fix $i$ and use $x'$ to denote $x + e_i$.
				
				First consider the case where $\xor(x) = 1,\ \maj(x) = 1$ and $x_i = 0$. Then we have $|x'| = m + 1$, $\xor(x') = 0$ and $\maj(x') = 1$. Hence, $\phi$ keeps $x$ constant and switches $x'$ to $y_{x'}$. We know that $y_{x'} \leq x'$ and is of weight $n - m - 1$ so $x$ and $y_{x'}$ must agree in at least $n - m - 2$ places. Hence, the distance between $x$ and $y_{x'}$ is bounded by $1 + m - (n - m -2) = 2m - n + 3$.
				
				Now consider the case $\xor(x) = 1,\ \maj(x) = 1$ and $x_i = 1$. In this case it is possible for $\maj(x') = 1$ or $\maj(x') = 0$ and we must consider these cases separately. Suppose $\maj(x') = 1$. Then $\phi$ keeps $x$ constant and switches $x'$ to $y_{x'}$ where $|y_{x'}| = n - m + 1$. We also have $y_{x'} \leq x + e_i \leq x$ so the stretch is $2m -n - 1$. If $\maj(x + e_i) = 0$, then $\phi$ keeps both $x$ and $x + e_i$ constant and the stretch is $1$.
				
				The other cases follow similarly.
			\end{proof}
		
		For a uniformly chosen $x$ and an arbitrarily chosen direction we have
		\begin{equation}\label{eqn:xortomaj}
			\mathbb{E}_{x,i} \left[ \dist \left(\phi(x), \phi(x + e_i)\right)  \right] \leq 3 + 2^{-n} \sum_{m = 0}^n \left| 2m - n \right| \binom{n}{m}.
		\end{equation}
		Let $S_n$ be the position a simple random walk on $\mathbb{Z}$ after $n$ steps (i.e. $S_n = \sum_{i=1}^n X_i$ where the $X_i$ are independent and take the value $1$ with probability 1/2 and the value $-1$ otherwise). Then $ \expec{|S_n|} = 2^{-n}\sum_{m = 0}^n \left| 2m - n \right| \binom{n}{m}$, and 
		applying standard results we get
		\[ \mathbb{E}_{x,i} \left[ \dist \left(\phi(x), \phi(x + e_i)\right)   \right] \leq 3 + \frac{n+1}{2^{n+1}} \binom{n + 1}{(n+1)/2} \sim \sqrt{\frac{2n}{\pi}}.\]
		\end{proof}
		Given that the average stretch for a mapping from $\xor$ to $\maj$ must be $\Omega(\sqrt{n})$, Rao and Shinkar ask if it is possible to relax the problem by increasing the size of the codomain and asking instead for a $O(1)$-Lipschitz injection. In the proof of Theorem \ref{thm:embed} below we give one such example.
		\begin{proof}[Proof of Theorem \ref{thm:embed}]
			Define $\phi(x) : \{0,1\}^n \to \{0,1\}^{2n+1}$ by \[ \phi(x) = \left( x^c, \xor(x), x\right) \]	where $x^c$ is the bitwise complement of $x$ (and where we have written $(a,b,c)$ for the concatenation of vectors $a$, $b$ and $c$). As we can read off $x$ from the last section, this is clearly a one-to-one mapping. If $x$ has Hamming weight $m$, $x^c$ has Hamming weight $n-m$, so overall $\phi(x)$ has Hamming weight $n + \xor(x)$. Hence, $\maj(\phi(x)) = \xor(x)$ for all $x \in \{0,1\}^n$.
			
			Suppose $x, y \in \{0,1\}^n$ are distinct. Then
			\begin{align*}
			\dist(\phi(x), \phi(y)) &= \dist(x^c, y^c) + \dist(\xor(x), \xor(y)) + \dist(x, y)\\
			&= 2\dist(x,y) + \xor(x+y)\\
			&\leq 3 \dist(x,y)
			\end{align*}
			and the map is 3-Lipschitz.
		\end{proof}
	\section{Bijections from $\dic$ to a random $f$}\label{sec:random}
	Let $f$ be a boolean function. We say a point $x \in \{0,1\}^n$ is a \emph{1 of $f$} if $f(x) = 1$ and similarly say $x$ is a \emph{0 of $f$} if $f(x) = 0$. We say a set $B$ is \emph{balanced} if $B$ contains an equal number of 1s and 0s of $f$, is \emph{positively imbalanced} if $B$ contains more 1s of $f$ than 0s of $f$ and \emph{negatively imbalanced} if $B$ contains more 0s of $f$ than $1$s of $f$. The \emph{imbalance} of $f$ over a set $B$ is the unsigned difference between the number of 1s of $f$ and $|B|/2$.
	
	Before we prove Theorem \ref{thm:random}, let us first sketch the ideas behind the proof without worrying about the technicalities.
	
	\begin{proof}[Sketch proof of Theorem \ref{thm:random}]
		We start by partitioning the hypercube $\{0,1\}^n$ into ``blocks'' which contain a large polynomial number of points (say $n^{\vara}$), have bounded diameter and contain an equal number of 1s and 0s of $\dic$. These blocks are small in comparison to the cube so, in a given block $B$, we expect the number of $1$s of $f$ to have a distribution similar to that of a binomial random variable with $|B|$ trials and success probability $1/2$. This means we expect every block to have an imbalance not much more than $n^{\vara/2}$ and a large proportion (tending to 1) of the blocks to have an imbalance not much less than $n^{\vara/2}$. We call a sequence of distinct blocks $B_1, \dots, B_k$ a \emph{block path} if, for $i= 1, \dots, k-1$, there is an edge between a vertex of $B_i$ and a vertex of $B_{i+1}$. In our construction, we take a large set of random block paths in the hypercube and use these to move the imbalance. If a block $B_1$ has too many $1$s of $f$, we find a block path $B_1, \dots, B_k$ from $B_1$ to a block $B_k$ with too many $0$s and map a $0$ of $\dic$ from $B_1$ to a $0$ of $f$ from $B_2$, a $0$ of $\dic$ from $B_2$ to a $0$ of $f$ from $B_3$, and so on. We also do the same with the $1$s but in the opposite direction. That is, we map a $1$ of $\dic$ from $B_2$ to a $1$ of $f$ from $B_1$, map a $1$ of $\dic$ from $B_3$ to a $1$ of $f$ from $B_2$, and so on. By doing this along enough paths, we can even out the sets and then arbitrarily match within them.
	
	For this to work we need to make sure that we don't try to map too many points to a block, which means ensuring there aren't too many paths through any single block. However, we also need to make sure there are enough paths between the blocks to spread the imbalance around. Once we have chosen our paths (at random) we construct a bipartite graph between the positive and negative blocks and find a subgraph with suitable degrees. This corresponds to the paths that will actually be used to move the imbalance around. The imbalance is not much more than $n^{\vara/2}$ in any block, so we don't need too many paths; at the same time, the imbalance is not much less than $n^{\vara/2}$ in most blocks, which ensures that there are enough random paths between imbalanced sets.
	\end{proof}
	
	Our proof of Theorem \ref{thm:random} uses the fact that, with high probability, there is a perfect matching in our random bipartite graph. The following lemma is an immediate consequence of a result of Erd\H{o}s and R\'enyi \cite{Erdos1963}. (The lemma can easily be strengthened, but this form is sufficient for our case.)
	
	\begin{lemma}\label{lem:bipartite}
		Let $A$ be a multiset of at least $2n \log n$ pairs $(i,j)$ each chosen uniformly and independently at random from $[n] \times [n]$. Let $G$ be the random bipartite graph with vertex set $\{v^+_1, \dots, v^+_n, v^-_1, \dots, v^-_n\}$, where $v^+_iv^-_j$ is an edge if and only if $(i,j) \in A$. Then the probability that $G$ contains a perfect matching tends to 1 as $n \to \infty$.
	\end{lemma}

	We will also make use of the following Chernoff bounds; see \cite{mitzenmacher2005probability} for a discussion and derivation of these bounds. 
	\begin{lemma}\label{lem:chernoff}
		Let $X \sim \bin(n,p)$ and $t > 0$. Then
		\begin{align*}
		\mathbb{P} \left( X \geq np + t \right) \leq e^{-2t^2/n} .
		\intertext{Also, for $0 \leq \eps \leq 1$ we have}
		\mathbb{P} \left(X\geq (1+\eps) np \right) \leq e^{-\eps^2 np /4}
		\intertext{and}
		\mathbb{P} \left(X \leq (1-\eps) np \right) \leq e^{-\eps^2 np /2}.
		\end{align*}
	\end{lemma}
	
	We are now ready to prove the main theorem. Throughout this proof we will use the constants $\vara$, $\varb$, $\varc$ and $\vard$ to which we assign explicit values only at the end of the proof.
	
	\begin{proof}[Proof of Theorem \ref{thm:random}]
		Suppose first that we have a mapping $\phi: \{0,1\}^n \to \{0,1\}^n$ from $\dic$ to $f$ such that $\dist(x, \phi(x)) \leq C$ for all $x \in \{0,1\}^n$.	Fix $x \in \{0,1\}^n$ and let $y = x + e_i$ for some $i$. Then 
		\[\dist(\phi(x), \phi(y)) \leq \dist(x, \phi(x)) + \dist(x,y) + \dist(y, \phi(y)) \leq 2C + 1 \]
		and the triangle inequality shows $\phi$ is $(2C+1)$-Lipschitz. A similar argument shows that $\phi^{-1}$ is also $(2C+1)$-Lipschitz and this proves the latter part of the theorem. Thus it remains to show that such a $\phi$ exists with high probability.

		We shall assume for the rest of the argument that $n$ is large enough for our estimates to hold. Let $\vara > 2$ be a constant. We start by partitioning the hypercube $\{0,1\}^n$ into ``blocks'' with diameter at most $4\vara$ which contain between $n^{\vara -2}$ and $4n^{\vara}$ elements, exactly half of which are 1s of $\dic$. Let $B$ be a maximal set of points in $\{0,1\}^{n-1}$ such that the pairwise distance between any two points is at least $2\vara$. For each point $v \in B$, we define the set of points $A_v$ to be the points in $\{0,1\}^{n-1}$ which are closer to $v$ than to any other point in $B$ (settling ties arbitrarily). As the set $B$ is maximal, the radius of $A_v$ must be less than $2\vara$, or else we could add a point to $B$. As the distance from $v$ to any other point in $B$ is at least $2\vara$, $A_v$ contains a Hamming ball of radius $\vara - 1$ centred at $v$ and at least $n^{\vara -2}$ points. If a set contains $N$ elements with  $N > 2n^{\vara}$, we arbitrarily split the set into $\floor{N/n^{\vara}}$ sets each containing between $n^{\vara}$ and $2n^{\vara}$ points.
		For each of our sets $A$, we define a corresponding ``block'' in $\{0,1\}^n$ to be $\{0,1\} \times A$. Clearly this gives a set $\mathcal{B}$ of blocks with the desired properties. Note that $2^n/(4n^{\vara})\leq |\mathcal{B}| \leq 2^n/n^{\vara -2}$.
		
		The following claim shows that, with probability tending to 1, no block $B \in \mathcal{B}$ has an imbalance much more than $n^{\vara/2}$.

	\begin{claim}\label{claim:binom}
		Provided that $b > 1/2$, with probability tending to 1, none of the sets $B \in \mathcal{B}$ have an imbalance of more than $n^{\vara/2 + \varb}$.
	\end{claim}
	\begin{proof}
		Say that a block $B$ is \emph{bad} under a function $f$ if $B$ has an imbalance greater than $n^{\vara/2 + \varb}$. Let $g: \{0,1\}^n \to \{0,1\}$ be a random boolean function where $g(x)$ is chosen independently and uniformly at random for each $x \in \{0,1\}^n$. Then the probability that $B$ is bad under $f$ is the same as the probability that $B$ is bad under $g$ conditioned on the event that $g$ is balanced, so
		\begin{align*}
		\mathbb{P} \left( B \text{ bad under $f$} \right) = \mathbb{P} \left(\left.  B \text{ bad under $g$} \right| g \text{ balanced} \right)
		\leq \frac{\mathbb{P} \left(B \text{ bad under $g$}  \right)}{\mathbb{P} \left( g \text{ balanced} \right)}.
		\end{align*}
		 The number of 1s of $g$ in $B$ is a binomial random variable so the first Chernoff bound in Lemma \ref{lem:chernoff} shows the probability that $B$ is bad under $g$ is at most $2e^{-2n^{\vara + 2\varb}/N} \leq 2e^{-n^{2\varb}/2}$. The probability that $g$ is balanced is $2^{-n} \binom{2^{n}}{2^{n-1}} \sim (\pi 2^{n-1})^{-1/2}$. Using the union bound over all blocks $B \in \mathcal{B}$, the probability that at least one of the sets $B$ is bad is at most $2^{3n/2} e^{-n^{2\varb}/2}$ for all sufficiently large $n$, and this is $o(1)$ provided $\varb > 1/2$.
	\end{proof}
	We will also need that most sets have an imbalance not much less than $n^{\vara/2}$.
	\begin{claim}\label{claim:binomSmall}
		Provided that $c > 2$, with probability tending to 1, the sum over $B \in \mathcal{B}$ of the imbalance of $B$ is at least $2^{n}/(8n^{\vara/2 + \varc})$.
	\end{claim}
	\begin{proof}
		Fix $B$ and let $|B| = N$. We can then write down the probability that the imbalance of $B$ is at most $n^{\vara/2  - \varc}$ as
		\[ \sum_{i = -n^{\vara/2 - \varc}}^{n^{\vara/2 - \varc}} \frac{ \binom{N}{N/2 + i} \binom{2^n - N}{2^{n-1} - N/2 - i}}{\binom{2^n}{2^{n-1}}}. \]
		Using $\binom{2n}{n} \sim \frac{4^n}{\sqrt{\pi n}}$ and that $N = o(2^n)$, the largest term ($i=0$) is $O(N^{-1/2})$. As $N \geq n^{\vara - 2}$, each term is $O(n^{- (\vara/2 - 1)})$ and the sum is $O(1/n^{\varc - 1})$. Let $X$ be the number of blocks with imbalance at most $n^{\vara/2 - \varc}$. Then,
		\[ \expec{X} =  O \left(\frac{|\mathcal{B}|}{n^{\varc - 1}} \right) \]
		and so, by Markov's Inequality, $\mathbb{P}\left( X \geq |\mathcal{B}|/n\right) = O(1/n^{\varc - 2})$. This is $o(1)$ provided $\varc > 2$, and in this case, with probability tending to 1, the sum of the imbalances is at least $(2^n/(4n^{\vara})) \cdot (1 - 1/n) \cdot n^{\vara/2 - \varc} \geq 2^{n}/(8n^{\vara/2 + \varc})$ for large $n$. 
	\end{proof}
	Now suppose that every block has an imbalance of at most $n^{\vara/2 + \varb}$ and that the total imbalance is at least $2^n/(8n^{\vara/2 + \varc })$.
	
    Independently sample $K = \lfloor{2^n/n^{\vard}}\rfloor$ uniformly random pairs of vertices in the hypercube and a random shortest path between them.
	\begin{claim}\label{claim:pathsHit}
		Provided that $a + 1 -d > 1$, with high probability, the number of paths in $K$ which intersect a given block $B \in \mathcal{B}$ is at most $8 n^{\vara + 1 - \vard}$.
	\end{claim}
	\begin{proof}
			By symmetry every point is equally likely to be on a random shortest path and, as any shortest path has at most $n$ vertices, the probability that a given vertex is on each sampled path is at most $n/2^n$. Using the union bound the probability that a given random path goes through a fixed $B \in \mathcal{B}$ is at most $n |B|/2^n$. As each path is sampled independently the number of paths through a fixed $B$ is dominated by a binomial distribution with $K$ trials and success probability $4n^{\vara + 1}/2^n$. Applying the second Chernoff bound from Lemma \ref{lem:chernoff} with $\eps = 1$, the probability that there are more than $8 n^{\vara + 1 - \vard}$ paths through a given block is at most $\exp( - n^{\vara + 1 - \vard}/2)$ and, taking the union bound, the probability that any $B$ intersects more than $8 n^{\vara + 1 - \vard}$ paths is at most $ \exp( - n^{\vara + 1 - \vard}/2) \cdot 2^n/n^{\vara - 2}$. This tends to 0 as $n \to \infty$ provided that $a + 1 - d > 1$.
	\end{proof}

	We now form a random bipartite graph $G$ with vertex sets $V^+$ and $V^-$ as follows. If $B \in \mathcal{B}$ has a positive imbalance of $i$, we add $i$ \emph{positive vertices} $v^B_1, \dots, v^B_{i}$ to $V^+$ and similarly if $B$ has a negative imbalance of $i$, we add $i$ \emph{negative vertices} to $V^-$. We generate the edge set of $G$ by examining each of our random paths in turn. Given a path, say between $v_1 \in B_1$ and $v_2 \in B_2$, we do the following:
	\begin{itemize}
		\item Discard the path with probability $1 - n^{2\vara - 4}/(|B_1|\cdot |B_2|)$. 
		\item If the path hasn't been discarded, independently choose two numbers $i,j$ in $[n^{\vara/2 + b}]$ and say that the edge $v_i^{B_1}v_j^{B_2}$ is present if both vertices exist and have opposite signs.
	\end{itemize}
	
	\begin{claim}
		The graph $G$ has a perfect matching $M$ with high probability provided 
			\begin{equation}\label{eqn:varbound}
				a/2 > 3b + 2c + d + 7.
			\end{equation}
	\end{claim}
	\begin{proof}
        The probability that a random path gives any particular edge between $V^+$ and $V^-$ is $p = 2(n^{\vara/2 - 2 - \varb}/2^n)^2$ and, in particular, every edge is equally likely. Hence, the number of edges in the multiset induced by the sampled random paths has a binomial distribution with $K$ trials and success probability $|V^+|^2 p$. By assumption $|V^+| \geq 2^n/(16n^{\vara/2 + \varc})$ and so $|V^+|^2p \geq (128n^{2 \varb + 2 \varc + 4})^{-1}$. Using the third Chernoff bound in Lemma \ref{lem:chernoff}, the probability the multiset contains at most $2^{n-9}/n^{ 2\varb  + 2 \varc + \vard + 4}$ pairs is at most $\exp ( - 2^{n-11}/n ^{2\varb  + 2 \varc + \vard + 4})$. There are at most $2^n/n^{\vara -2}$ blocks each with imbalance at most $n^{\vara/2 + \varb}$ so $|V^+| \leq 2^n/n^{\vara/2 - \varb - 2}$. Hence, $2|V^+| \log |V^+| \leq 2^{n+1}/n^{\vara/2 - \varb - 3}$. From (\ref{eqn:varbound}), for large enough $n$,
		\begin{equation*}
			2^{n+1}/n^{\vara/2 - \varb -3} \leq 2^{n-9}/n^{2\varb + 2 \varc + \vard + 4}, 
		\end{equation*}
		and we have that $G$ contains a perfect matching $M$ with high probability by Lemma \ref{lem:bipartite}.
	\end{proof}

	Given the matching $M$ we can greedily construct the map $\phi$. We start with all vertices as both unset and unused. During the construction we will let $\phi(x) = y$ for some $x, y \in \{0,1\}^n$ and we will then say that $x$ is \emph{set} and $y$ is \emph{used}. The vertices for which we still have to define $\phi$ are exactly the unset vertices, and the vertices which are not yet in the image of $\phi$ are exactly the unused vertices.
	
	Each edge in $M$ corresponds to a path in the hypercube and this induces a walk between two blocks in the block decomposition. By removing any loops, we can say the edge corresponds to a path $P = B_{i_1}\dots B_{i_k}$ in the block decomposition where $B_{i_1}$ has positive imbalance and $B_{i_k}$ has negative imbalance. For $j = 1, \dots, k-1$, we choose any point $x \in B_{i_j}$ with $\dic(x) = 0$ which is currently unset and any point $y \in B_{i_{j+1}}$ with $f(y)  = 0$ which is currently unused and set $\phi(x) = y$, setting $x$ and using $y$. We also choose any point in $x' \in B_{i_{j+1}}$ with $\dic(x) = 1$ which is currently unset and any point $y' \in B_{i_j}$ with $f(y) = 1$ and set $\phi(x') = y'$. In order to guarantee there will always be enough points to do the above, we will enforce that $8n^{\vara + 1 - d} \leq n^{\vara - 2}$. 
	
	If $B$ has positive imbalance $i$, there are $i$ paths starting at $B$ (those corresponding to the edges in the matching using $v^B_1, \dots, v^B_i$) and hence, after processing all the paths which start at $B$, the number of unset 1s of $\dic$ in $B$ is still $|B|/2$ while the number of $0$s is now $|B|/2 - i$. Similarly the number of unused 1s of $f$ in $B$ is now $|B|/2$ while the number of $0$s is still $|B|/2 - i$. The paths for which $B$ is an internal block  reduce all four quantities by 1 so, after processing all the paths, the number of 1s and 0s of $\dic$ which are unset and the number of 1s and 0s of $f$ which are unused, are equal. This means we can greedily complete $\phi$ inside $B$ by taking any unset $x$ and any unused $y$ with $\dic(x) = f(y)$ and setting $\phi(x) = y$. A similar argument works when $B$ is balanced or has negative imbalance.
	
	It is clear that, provided nothing goes wrong, this construction gives a bijection $\phi: \{0,1\}^n \to \{0,1\}^n$ which maps $\dic$ to $f$. If $x \in B$, then either $\phi(x) \in B$ or $\phi(x) \in B'$ for some $B'$ for which there is an edge between $B$ and $B'$ in the hypercube. As the diameter of a block is bounded by $4\vara$, the distance between $x$ and $\phi(x)$ is bounded by $8\vara + 1$. 
	
	This construction can fail if any block has too large an imbalance, the total imbalance is too small, there are too many paths through a block, there aren't enough edges in the multiset or there isn't a matching in $G$; but all of these events happen with probability tending to 0 (for suitable choices of $\vara$, $\varb$, $\varc$ and $\vard$).
	
	To complete the proof it remains to find suitable values for the constants $\vara, \varb, \varc$ and $\vard$. From Claim \ref{claim:binom} we must have $b > 1/2$ and from Claim \ref{claim:binomSmall} we must have $c > 2$. We will also need $8n^{\vara + 1 - d} \leq n^{\vara - 2}$, which follows for large enough $n$ provided $d > 3$. Finally, we require $a > 2$, $a + 1 - d > 1$ and, from Claim \ref{claim:pathsHit}, that $a/2 > 3b + 2c + d + 7$. All these constraints can be satisfied by taking $a = 42$, $b = 1$, $c = 3$ and $d = 4$.
	\end{proof}
	
	\section{Open Problems}\label{sec:problemsLip}
	We saw in Theorem \ref{thm:dictoxor} that it is not possible for a map from $\dic$ to $\xor$ to be $O(1)$-Lipschitz, $O(1)$-local and also have an inverse that is $O(1)$-Lipschitz. On the other hand we saw that maps exist if we drop the condition on the inverse or on the locality. Can we drop the first condition while keeping the other two?
	\begin{problem}
		Is there a mapping $\phi: \{0,1\}^n \to \{0,1\}^n$ from $\dic$ to $\xor$ such that each output bit depends on $O(1)$ input bits and such that $\phi^{-1}$ is $O(1)$-Lipschitz?
	\end{problem}
	Corollary \ref{cor} shows that, with high probability, if $A,B \subseteq \{0,1\}^n$ are two independent uniformly random sets with $|A| = |B| = 2^{n-1}$, there exists a $O(1)$-bi-Lipschitz bijection $\phi: \{0,1\}^n \to \{0,1\}^n$ such that $\phi(A) = B$,  but we made no attempt to optimise the constant. We expect a much smaller constant to hold in this corollary (and in Theorem \ref{thm:random}), possibly even as small as 2 or 3. How small could the constant be?

	\begin{problem}
		Let $A,B \subseteq \{0,1\}^n$ be two independent uniformly random sets with $|A| = |B| = 2^{n-1}$. What is the smallest constant $C$ such that, with high probability, there is a $C$-bi-Lipschitz bijection $\phi:\{0,1\}^n \to \{0,1\}^n$ such that $\phi(A) = B$?
	\end{problem}

	There are many interesting variations of this problem. For example, what happens if we let $|A| = |B| = f(n)$ for some function $f(n)$? 
		
	\begin{problem}
		Let $A,B \subseteq \{0,1\}^n$ be two independent uniformly random sets with $|A| = |B| = f(n)$. For what functions $f: \mathbb{N} \to \mathbb{N}$ does there exist a $C$ such that, with high probability, there is a $C$-bi-Lipschitz bijection $\phi:\{0,1\}^n \to \{0,1\}^n$ with $\phi(A) = B$? For example, what happens when $f(n) = \Theta \left( 2^n / n \right)$?
	\end{problem}

	A closely related problem concerns colourings. If we view the partition of $\{0,1\}^n$ into 2 parts ($A$ and $A^c$) as a partition into 2 colour classes, then we can view mappings as relabellings of the cube which respect two balanced 2-colourings. What happens if we instead colour the cube with $k$ colours? If $k$ is a constant, then a modification of the argument used to prove Theorem \ref{thm:random} will show that there is an $O(1)$-bi-Lipschitz mapping with high probability. But what happens if we let $k \to \infty$ as $n \to \infty$?
	
	We say that a partition of $\{0,1\}^n$ into $k$ parts $A_1, \dots, A_k$ is a \emph{balanced partition} if $|A_1| \leq |A_2| \leq \dotsb |A_k| \leq |A_1| + 1$.
	 
	 \begin{problem}
	 	Let $A_1, \dots, A_k$ and $B_1, \dots, B_k$ be two independent uniformly random balanced partitions of $\{0,1\}^n$ into $k(n)$ parts. For what functions $k : \mathbb{N} \to \mathbb{N}$ does there exist a constant $C$ such that, with high probability, there is a $C$-bi-Lipschitz bijection $\phi : \{0,1\}^n \to \{0,1\}^n$ with $\phi(A_i)= B_i$ for all $i$? How small can $C$ be?
	 \end{problem}
 
 	The problems above have been concerned with functions on the hypercube but we can ask similar questions for functions on other graphs and, in particular, Cayley graphs. For example, what happens for $\mathbb{Z}_n \times \mathbb{Z}_n$ generated by the elements $(1,0)$ and $(0,1)$? In this case, we say that a map $\phi : \mathbb{Z}_n \times \mathbb{Z}_n\to \mathbb{Z}_n \times \mathbb{Z}_n$ is \emph{$C$-Lipschitz} if, for every $x,y \in V$, $d(\phi(x), \phi(y)) \leq C d(x,y)$ where $d$ is the graph distance in the Cayley graph. The proof of Theorem \ref{thm:random} relied on a decomposition of $\{0,1\}^n$ into ``blocks'', each with bounded radius but containing a large polynomial number of points. In $\mathbb{Z}_n \times \mathbb{Z}_n$ any subset with bounded radius must contain a bounded number of points and the ``blocks'' won't tend towards being relatively balanced.
 \begin{problem}
	Let $A,B \subseteq \mathbb{Z}_n \times \mathbb{Z}_n$ be two independent uniformly random sets with $|A| = |B| = \ceil{n^2/2}$. What is the best $C = C(n)$ such that, with high probability, there is a $C$-bi-Lipschitz bijection $\phi:\mathbb{Z}_n \times \mathbb{Z}_n \to \mathbb{Z}_n \times \mathbb{Z}_n$ with $\phi(A) = B$?
 \end{problem}

 	In a different direction, we observe that there is a natural symmetric measure of the difference between two boolean functions. Define the \emph{Lipschitz constant} of a mapping $\phi$ by  \[ \lipsize{\phi} = \inf\left\{ C \geq 0 : \|\phi(x) - \phi(y)\| \leq C\|x-y\| \text{ for all } x,y \in \{0,1\}^n \right\}, \] 
 	and define the \emph{dissimilarity} between two boolean functions $f$ and $g$ by
 	 	\[\dlip(f,g) := \inf \left\{ \lipsize{\phi} \cdot \lipsize{\phi^{-1}} : \phi \text{ is a mapping from $f$ to $g$} \right\}. \]
 
 	The 2-bi-Lipschitz mapping between $\dic$ and $\xor$  we saw earlier implies that $\dlip(\dic, \xor) \leq 4$, and Theorem \ref{thm:random} states that, if $f$ and $g$ are uniformly random balanced boolean functions, $\dlip(f,g) = O(1)$ with high probability. On the other hand we know that any mapping from $\dic$ to $\maj$ is not $C$-Lipschitz for any $C < n/2$ so $\dlip(\dic, \maj) = \Omega(n)$. Rao and Shinkar \cite{RAO2018} give a $O(1)$-Lipschitz mapping from $\maj$ to $\dic$, so we can strengthen this to $\dlip(\dic, \maj) = \Theta(n)$. In general we know that every mapping is $n$-Lipschitz which gives a trivial upper bound of $n^2$, but is it possible to do better?

 	\begin{problem}
 		Is there a constant $\alpha < 2$ such that $\dlip(f,g) = O(n^\alpha)$ for all balanced boolean functions $f$ and $g$?
 	\end{problem}

	Finally, we note that the paper by Rao and Shinkar is motivated in part by a problem from Benjamini, Cohen and Shinkar in \cite{Benjamini2016} which remains open and seems very interesting.
	
	\begin{problem}
		Is there a set $A \subseteq \{0,1\}^{n}$ of size $|A| = 2^{n-1}$ such that any bijection from $\{0,1\}^{n-1}$ to $A$ has average stretch $\omega(1)$?
	\end{problem}

\subsubsection*{Acknowledgements}
	We would like to thank an anonymous referee for their helpful comments.
	\bibliographystyle{abbrv-bold}
	\bibliography{boolean_functions}
	
\end{document}